\documentclass[11pt]{amsart}

\usepackage[T1]{fontenc}
\usepackage{lmodern}
\usepackage[utf8]{inputenc}
\usepackage{amsmath,amssymb,amsthm,mathtools}
\usepackage{enumitem}
\usepackage{tikz}
\usetikzlibrary{arrows.meta,positioning,calc}
\usepackage[numbers,sort&compress]{natbib}
\usepackage{hyperref}
\usepackage{microtype}

\hypersetup{colorlinks=true,linkcolor=blue,citecolor=blue,urlcolor=blue}

\theoremstyle{plain}
\newtheorem{theorem}{Theorem}[section]
\newtheorem{lemma}[theorem]{Lemma}
\newtheorem{proposition}[theorem]{Proposition}
\newtheorem{corollary}[theorem]{Corollary}
\theoremstyle{definition}
\newtheorem{definition}[theorem]{Definition}
\newtheorem{remark}[theorem]{Remark}

\newcommand{\toq}{\to_q}
\newcommand{\Cyc}{\overrightarrow{C}}
\newcommand{\Id}{I}
\newcommand{\HH}{\mathcal H}
\newcommand{\DiCycles}{\mathsf{DiCycles}}
\newcommand{\Aset}{\mathsf A}
\newcommand{\Bset}{\mathsf B}
\newcommand{\Eset}{\mathsf E}

\title[The Quantum Homomorphism Orders of Graphs Are Universal]{The Quantum Homomorphism Orders of Graphs Are Universal}
\author{Yangjing Long}
\address{School of Mathematics and Statistics, Central China Normal University, 152 Luoyu Road, Hongshan District, Wuhan, Hubei 430079, China}
\email{yangjing@ccnu.edu.cn}

\begin{document}

\begin{abstract}
Quantum graph homomorphisms, introduced by Man\v{c}inska and Roberson, form a natural quantum relaxation of classical graph homomorphisms.  Since every classical homomorphism is quantum but not conversely, this relaxation can create new comparabilities and could in principle simplify the order by collapsing antichains and other incomparability patterns.  We prove that this does not happen.  In the order-theoretic sense, the quantum homomorphism orders of graphs remain as complicated as possible among countable partial orders: every countable partial order embeds into the quantum homomorphism quasi-order of finite directed graphs and also into that of finite undirected graphs, and hence into the corresponding quotient partial orders.

The result is also useful as a construction principle.  For every finite poset $P$, our proof gives finite undirected graphs $G_p$, $p\in P$, such that
\[
   p\le_P q
   \quad\Longleftrightarrow\quad
   G_p\toq G_q.
\]
Thus any prescribed finite pattern of quantum homomorphisms and non-homomorphisms can be implemented by ordinary finite graphs, including antichains, chains, diamonds, crowns, Boolean lattices, and arbitrary finite order patterns.  The finite construction is explicit: encode the poset by prime divisibility among directed cycles, and then replace each directed edge by one fixed undirected endpoint-forcing indicator.

For directed graphs, the proof uses the classical universality of the homomorphism order on finite disjoint unions of clockwise directed cycles, together with the fact that quantum homomorphisms between such directed cycles coincide with classical homomorphisms.  For undirected graphs, the main technical ingredient is a finite ordered indicator whose terminals are quantum endpoint-forcing.  In a classical indicator argument one controls the endpoint images $f(a)$ and $f(b)$.  In the quantum setting there is no single-valued image map; the replacement is the operator statement that
\[
   F_{a,p}F_{b,q}=0
\]
for every illegal ordered terminal pair $(p,q)$.  Thus the gadget converts pointwise endpoint forcing into joint-support forcing for projections.  Replacing each directed edge by this indicator embeds the directed-cycle order into the quantum homomorphism order of finite undirected graphs.  Thus quantum homomorphisms enlarge the classical homomorphism relation, but they do not collapse its order-theoretic complexity.  The proof is finite and combinatorial, and uses only the projection relations in the definition of a quantum homomorphism.
\end{abstract}

\maketitle

\section{Introduction}

For finite graphs $X$ and $Y$, a homomorphism $X\to Y$ is an adjacency-preserving map from $V(X)$ to $V(Y)$.  The existence of a homomorphism is a reflexive and transitive relation, and hence induces a partial order after quotienting by homomorphic equivalence.  This order is very rich: it is universal, in the sense that every countable partial order occurs as a suborder.  We use this term in the same order-theoretic sense throughout the paper.

The universality of graph homomorphism orders is a classical theme.  Hedrl\'{i}n proved that the homomorphism order of finite graphs is universal \cite{HedrlinUniversal}; see also the categorical representation theory of Pultr and Trnkov\'{a} \cite{PultrTrnkovaBook} and the monograph of Hell and Ne\v{s}et\v{r}il \cite{HellNesetril}.  Later work showed that universality persists in sparse or restricted classes.  Hubi\v{c}ka and Ne\v{s}et\v{r}il proved universality for finite oriented paths and for oriented trees and related simple graph classes \cite{HubickaNesetrilPaths,HubickaNesetrilTrees}.  Fiala, Hubi\v{c}ka and the author gave a divisibility-based universality argument using oriented cycles \cite{FialaHubickaLong}, and the stronger fractal property of the homomorphism order was subsequently treated by Fiala, Hubi\v{c}ka, Long and Ne\v{s}et\v{r}il \cite{FialaHubickaLongNesetril}.  Luchnikov, Wittebol and Zuiddam recently proved an analogous universality theorem for the asymptotic cohomomorphism order arising in the asymptotic spectrum of graphs \cite{LuchnikovWittebolZuiddam}.

Other homomorphism-type orders provide useful comparisons.  For example, Naserasr, Sen and Sopena studied the homomorphism order of signed graphs and proved lattice properties of that order \cite{NaserasrSenSopenaSignedOrder}; see also the update on signed-graph homomorphisms by Naserasr, Sopena and Zaslavsky \cite{NaserasrSopenaZaslavskySignedUpdate}.  These signed-graph results are classical rather than quantum, but they illustrate how graph-homomorphism orders change when the underlying category of graphs is modified.

Quantum homomorphisms were introduced by Man\v{c}inska and Roberson \cite{MRQuantumHom} as the existence of perfect quantum strategies for the graph homomorphism game.  Their work placed quantum coloring, quantum independence and quantum clique numbers in a common homomorphism framework, proved several structural properties, and studied the relationship between the quantum homomorphism order and the classical homomorphism order, including order-theoretic features of the former.  Roberson's thesis further developed the order-theoretic viewpoint, including a description of the quantum homomorphism order via measurement graphs \cite{RobersonThesis}.  This naturally leads to the question raised in \cite{MRQuantumHom}: which order-theoretic properties of the classical homomorphism order survive under the quantum relaxation?

It is important that this question is not a formal consequence of the classical universality theorem.  Every classical homomorphism is a quantum homomorphism, so
\[
   G\to H \quad\Longrightarrow\quad G\toq H.
\]
The converse is false in general: quantum homomorphisms may create comparabilities which do not exist classically, as already happens in quantum coloring.  From the order-theoretic viewpoint this relaxation is potentially destructive.  New quantum arrows could collapse antichains and other configurations of incomparability, so the universality of the classical homomorphism order does not by itself imply the universality of the quantum homomorphism order.  The central issue is whether the passage from classical to quantum homomorphisms preserves or destroys the order-theoretic complexity of graph homomorphisms.  Our answer is that this complexity survives: quantum homomorphisms add many new arrows, but not enough arrows to destroy universality.

The paper \cite{MRQuantumHom} has since become a basic reference for several related directions.  On the coloring side, the quantum chromatic number had already been studied by Cameron, Montanaro, Newman, Severini and Winter, who related it to clique number and orthogonal representations and gave an 18-vertex example with classical chromatic number $5$ and quantum chromatic number $4$ \cite{CameronMontanaroNewmanSeveriniWinter}.  Man\v{c}inska and Roberson later exhibited further unexpected phenomena in quantum colorings, including small separations between classical and quantum chromatic number \cite{MancinskaRobersonOddities}.  Man\v{c}inska, Roberson and Varvitsiotis studied the decidability of perfect entangled strategies for nonlocal games and showed that independent-set games are representative of difficult instances \cite{MancinskaRobersonVarvitsiotis}.  Ortiz and Paulsen studied quantum graph homomorphisms through operator systems and associated universal $C^*$-algebras, and introduced an operator-system notion of quantum core \cite{OrtizPaulsen}.  Roberson gave conic formulations of graph homomorphism relations, including quantum and related relaxations \cite{RobersonConic}.  Helton, Meyer, Paulsen and Satriano developed the algebraic approach to synchronous games and graph coloring games \cite{HeltonMeyerPaulsenSatriano}.  Atserias, Man\v{c}inska, Roberson, \v{S}\'amal, Severini and Varvitsiotis introduced quantum and non-signalling graph isomorphism games, linking graph isomorphism, nonlocal games and noncommutative feasibility problems \cite{AtseriasEtAl}.  Lupini, Man\v{c}inska and Roberson connected these games to quantum permutation groups and quantum automorphism groups \cite{LupiniMancinskaRoberson}.  Man\v{c}inska and Roberson later proved that quantum isomorphism is equivalent to equality of homomorphism counts from all planar graphs \cite{MancinskaRobersonPlanar}, showing that homomorphism-counting techniques also capture quantum symmetries.

There has also been renewed activity in the last few years around homomorphism games, homomorphism indistinguishability, and quantum relaxations of constraint satisfaction.  Brannan, Ganesan and Harris introduced the quantum-to-classical graph homomorphism game \cite{BrannanGanesanHarris}.  Harris proved that graph homomorphism games are universal for synchronous games and used this to study the quantum coloring problem \cite{HarrisUniversalityGames}.  Ciardo related quantum advantage in homomorphism-type CSP tasks to the polymorphism-minion framework from algebraic CSP complexity \cite{CiardoQuantumAdvantage}, and studied the possible size of the gap between classical and quantum chromatic number, including a quantum adjunction theorem for Pultr functors \cite{CiardoQuantumChromaticGap}.  Banakh, Ciardo, Kozik and Tu\l{}owiecki investigated classical simulation of quantum CSP strategies \cite{BanakhCiardoKozikTulowiecki}, while Ciardo, Joubert and Mottet introduced quantum polymorphisms and commutativity gadgets for quantum CSPs \cite{CiardoJoubertMottet}.  Karamlou studied quantum relaxations of CSP and structure isomorphism using the quantum monad viewpoint \cite{KaramlouQuantumRelaxations}.  Kornell and Lindenhovius recently constructed quantum graphs of homomorphisms in a closed symmetric monoidal category of quantum graphs; their construction is nonempty precisely when the corresponding homomorphism game has a perfect quantum strategy \cite{KornellLindenhovius}.  Homomorphism-counting viewpoints have also led to semidefinite-programming characterizations: Roberson and Seppelt related the Lasserre hierarchy for graph isomorphism to homomorphism indistinguishability \cite{RobersonSeppeltLasserre}, and Kar, Roberson, Seppelt and Zeman gave the analogous NPA hierarchy characterization for quantum isomorphism \cite{KarRobersonSeppeltZeman}.  These works use notions of quantum homomorphism in algebraic, categorical, computational, and nonlocal-game settings.  Our use of the word ``universal'' is different from the universality of homomorphism games in \cite{HarrisUniversalityGames}: here universality means embedding every countable partial order into the quantum homomorphism order.

Our main result is that the quantum homomorphism order retains this fundamental richness of the classical order: arbitrary countable patterns of comparability and incomparability still occur even after the quantum relaxation.

\begin{theorem}\label{thm:main}
The quasi-order induced by quantum homomorphisms on finite directed graphs is countably universal, and so is the quasi-order induced by quantum homomorphisms on finite undirected graphs.  Consequently, after quotienting by quantum homomorphic equivalence, the corresponding partial orders are countably universal.
\end{theorem}

The paper has two closely related contributions.  First, it proves that quantum homomorphism orders retain the universality of the classical homomorphism order, despite the additional comparabilities created by quantum homomorphisms.  Second, it gives a projection-level replacement for the endpoint-forcing step used in classical indicator and arrow constructions.  The ordered indicator constructed below does not assert that two terminals have specified images; instead, it shows that the joint support of the two terminal projections vanishes on every illegal ordered pair.  This is the mechanism that lets directed information be encoded inside undirected quantum homomorphism constructions.

Let us spell out the meaning of the theorem.  Universality here means countable universality: for every countable partially ordered set $(P,\le_P)$ there is an injective assignment
\[
   p\longmapsto G_p
\]
from $P$ to finite graphs such that
\[
   p\le_P q
   \quad\Longleftrightarrow\quad
   G_p\toq G_q.
\]
Thus the quantum homomorphism order contains arbitrary countable configurations of comparability and incomparability.  In particular, every finite partial order can be found exactly inside the quantum homomorphism order of finite graphs.

Let us also clarify the role of quotienting and cores.  Classically, one often passes from the homomorphism quasi-order to a partial order by quotienting by homomorphic equivalence; every equivalence class then has a unique core.  For quantum homomorphisms one can define quantum homomorphic equivalence by
\[
   G\equiv_q H
   \quad\Longleftrightarrow\quad
   G\toq H\text{ and }H\toq G.
\]
A corresponding theory of quantum cores is more delicate, and we do not use it.  Our embeddings are proved directly at the level of the quasi-order by an if-and-only-if statement
\[
   G_p\toq G_q
   \quad\Longleftrightarrow\quad
   p\le_P q.
\]
Consequently, if $p\ne q$ in the embedded partial order, then the corresponding graphs are not quantum-homomorphically equivalent: otherwise both $G_p\toq G_q$ and $G_q\toq G_p$ would hold, forcing both $p\le_P q$ and $q\le_P p$.  Thus the embedding descends to an induced suborder of the quotient partial order.  No uniqueness theorem for quantum cores is required.

For finite partial orders this can be made completely explicit.  Given a finite poset $P$, choose distinct primes $\pi_r$, one for each $r\in P$, and put
\[
   N_p=\prod_{r:\,p\le_P r}\pi_r.
\]
Then
\[
   p\le_P q
   \quad\Longleftrightarrow\quad
   N_q\mid N_p.
\]
Since $\Cyc_m\to\Cyc_n$ if and only if $n\mid m$, the graphs obtained below by replacing each arc of $\Cyc_{N_p}$ with our endpoint-forcing indicator satisfy
\[
   p\le_P q
   \quad\Longleftrightarrow\quad
   \Phi(\Cyc_{N_p})\toq \Phi(\Cyc_{N_q}).
\]
For example, distinct primes give antichains of arbitrary size, powers of a prime give chains of arbitrary length, and the diamond poset can be represented by the four cycle lengths $210,6,10,2$.  Thus, for finite posets, the theorem is not merely existential: it gives a direct recipe for building finite undirected graphs with any prescribed pattern of quantum homomorphisms and non-homomorphisms.  This provides a concrete source of test configurations and gadgets---chains, antichains, diamonds, crowns, Boolean lattices, and arbitrary finite posets can all be realized exactly inside the quantum homomorphism order.  This illustrates that the passage from classical to quantum homomorphisms does not collapse the order-theoretic complexity of graph homomorphisms: the quantum order is still as complicated as a countable partial order can be.

The directed part is conceptually simple.  If $\Cyc_n$ denotes the clockwise directed cycle of length $n$, then
\[
   \Cyc_n\toq \Cyc_m
   \quad\Longleftrightarrow\quad
   \Cyc_n\to \Cyc_m
   \quad\Longleftrightarrow\quad
   m\mid n.
\]
It follows that the quantum order on finite disjoint unions of directed cycles agrees with the classical order.  Since the latter order is universal \cite{FialaHubickaLong}, the directed quantum homomorphism order is universal.

The undirected part is the main technical point of the paper.  Classical indicator and arrow constructions do not automatically transfer to the quantum setting: quantum homomorphisms are given by projections rather than set-valued maps, may split over orthogonal subspaces, and support-image intuition can be misleading.  Thus an undirected replacement gadget must forbid not only unintended classical images, but also unintended quantum correlations.  The endpoint-forcing indicator below is the mechanism that prevents unwanted quantum arrows and hence protects the incomparabilities needed for universality.  To avoid set-theoretic image arguments, we construct an ordered undirected indicator $J=J(a,b)$ with a quantum endpoint-forcing property.  For every irreflexive directed graph $H$ and every quantum homomorphism
\[
   F:J\toq H*J,
\]
where $H*J$ is obtained by replacing every arc $u\to v$ of $H$ by a copy of $J$ with $a$ identified with $u$ and $b$ identified with $v$, one has
\[
   F_{a,p}F_{b,q}=0
\]
unless $p,q$ are the ordered terminal vertices of a single copy of $J$ in $H*J$.

The point of the endpoint-forcing condition is the following.  In a classical indicator proof one often argues that, if a homomorphism $f:J\to H*J$ satisfies $f(a)=p$ and $f(b)=q$, then $(p,q)$ must be a legal ordered terminal pair.  A quantum homomorphism has no such single-valued map $f$; it is a family of projections $F_{x,y}$.  The appropriate replacement is therefore not a statement about endpoint images, but the vanishing statement
\[
   F_{a,p}F_{b,q}=0
\]
for illegal pairs.  Equivalently, the two terminal measurements have no joint support on such a pair.  This is the feature that allows the classical directed-cycle embedding to survive the quantum relaxation.

The construction of $J$ is finite and explicit.  Its spine has color sequence
\[
   E,A,A,B,B,A,E,
\]
where the ``colors'' $A,B,E$ are implemented by ordinary graph gadgets rather than by additional structure.  The key lemma is a quantum color-localization lemma showing that the projections corresponding to an $A$-root, a $B$-root, or an $E$-root vanish on target vertices not carrying the same gadget.  Thus a nonzero summand in the expansion of $F_{a,p}F_{b,q}$ forces a color-preserving walk of type $E,A,A,B,B,A,E$ in $H*J$, and the combinatorics of the spine implies that such a walk is necessarily the ordered spine of a single copy.

With this indicator in hand, the universality of the undirected order follows by embedding the classical order on finite disjoint unions of directed cycles into the quantum order of finite undirected graphs.  In this sense, the theorem provides an explicit encoding device for constructing prescribed order patterns inside the quantum homomorphism order of finite undirected graphs.  The proof is entirely finite and combinatorial, and uses only the projection relations in the definition of a quantum homomorphism.

\section{Quantum homomorphisms}

All graphs and digraphs in this paper are finite and loopless.  Undirected graphs are simple.  Directed graphs have no loops; parallel arcs will not be needed.  All Hilbert spaces appearing in quantum homomorphisms are finite-dimensional and nonzero.

\subsection{Definitions}

Let $X$ and $Y$ be finite undirected graphs.  A \emph{quantum homomorphism} from $X$ to $Y$ is a family of orthogonal projections
\[
   \{E_{x,y}:x\in V(X),\ y\in V(Y)\}
\]
on a nonzero finite-dimensional Hilbert space $\HH$ such that
\begin{align}
   \sum_{y\in V(Y)} E_{x,y} &= \Id &&\text{for every }x\in V(X), \label{eq:complete}\\
   E_{x,y}E_{x,y'} &=0 &&\text{whenever }y\ne y', \label{eq:samevertex}\\
   E_{x,y}E_{x',y'} &=0 &&\text{whenever }xx'\in E(X)\text{ and }yy'\notin E(Y). \label{eq:edgezero}
\end{align}
We write $X\toq Y$ when such a family exists.  This is the projective-measurement formulation of quantum graph homomorphisms from \cite{MRQuantumHom}.  Since $E_{x,y}$ and $E_{x,y'}$ are orthogonal projections for $y\ne y'$ and their sum is the identity, each set $\{E_{x,y}:y\in V(Y)\}$ is a projective measurement.

For directed graphs we use the same definition, replacing the edge condition \eqref{eq:edgezero} by
\[
   E_{x,y}E_{x',y'}=0
   \quad\text{whenever }(x,x')\in A(X)\text{ and }(y,y')\notin A(Y).
\]
The relation $\toq$ is reflexive and transitive, so it induces a quasi-order; the corresponding partial order is obtained by quotienting by the equivalence relation generated by $X\toq Y$ and $Y\toq X$.

\subsection{The walk lemma}

The following elementary lemma is the main tool used throughout the paper.

\begin{lemma}[Walk lemma]\label{lem:walk}
Let $X\toq Y$ be a quantum homomorphism, represented by projections $E_{x,y}$.  Suppose there is a directed walk of length $\ell$ from $x$ to $x'$ in $X$, but no directed walk of length $\ell$ from $y$ to $y'$ in $Y$.  Then
\[
   E_{x,y}E_{x',y'}=0.
\]
The same statement holds for undirected graphs, with walks interpreted in the undirected sense.
\end{lemma}

\begin{proof}
For $\ell=0$ the claim is exactly \eqref{eq:samevertex}; for $\ell=1$ it is \eqref{eq:edgezero}.  Suppose the statement holds for walks of length $\ell$.  Let
\[
   x=x_0,x_1,\ldots,x_\ell,x_{\ell+1}=x'
\]
be a walk of length $\ell+1$ in $X$.  Insert the resolution of the identity at $x_\ell$:
\[
   E_{x,y}E_{x',y'}
   =\sum_{z\in V(Y)} E_{x,y}E_{x_\ell,z}E_{x',y'}.
\]
If there is no walk of length $\ell$ from $y$ to $z$, then the induction hypothesis gives $E_{x,y}E_{x_\ell,z}=0$.  If there is such a walk from $y$ to $z$, then, since there is no walk of length $\ell+1$ from $y$ to $y'$, there is no edge from $z$ to $y'$; hence $E_{x_\ell,z}E_{x',y'}=0$ by the edge-zero condition.  Thus every summand is zero.
\end{proof}

\section{Directed cycles}

Let $\Cyc_n$ denote the clockwise directed cycle on $n$ vertices, with arcs $i\to i+1\pmod n$.

\begin{proposition}\label{prop:cycles}
For $m,n\ge 2$,
\[
   \Cyc_n\toq \Cyc_m
   \quad\Longleftrightarrow\quad
   \Cyc_n\to \Cyc_m
   \quad\Longleftrightarrow\quad
   m\mid n.
\]
\end{proposition}

\begin{proof}
If $m\mid n$, the map $i\mapsto i\pmod m$ is a classical homomorphism $\Cyc_n\to\Cyc_m$, and hence gives a quantum homomorphism.

Conversely, suppose $\Cyc_n\toq \Cyc_m$, represented by projections $E_{u,v}$.  Fix $u\in V(\Cyc_n)$.  Since $\sum_v E_{u,v}=\Id$, there exists $v$ with $E_{u,v}\ne 0$.  There is a directed closed walk of length $n$ from $u$ to $u$ in $\Cyc_n$.  If there were no directed closed walk of length $n$ from $v$ to $v$ in $\Cyc_m$, Lemma \ref{lem:walk} would give $E_{u,v}^2=0$, a contradiction.  A directed closed walk of length $n$ exists in $\Cyc_m$ if and only if $m\mid n$.
\end{proof}

We also need the standard reduction for disjoint unions.

\begin{lemma}\label{lem:connected-to-component}
Let $X$ be a weakly connected directed graph, and let $Y=Y_1\sqcup\cdots\sqcup Y_t$ be a disjoint union of directed graphs.  If $X\toq Y$, then $X\toq Y_i$ for some $i$.
\end{lemma}

\begin{proof}
Let $E_{x,y}$ represent a quantum homomorphism $X\toq Y$.  For $x\in V(X)$ and $i\in\{1,\ldots,t\}$, set
\[
   P_i^x=\sum_{y\in V(Y_i)}E_{x,y}.
\]
For fixed $x$, the projections $P_i^x$ are pairwise orthogonal and sum to $\Id$.  If $xx'$ is an arc of $X$, then there are no arcs between distinct components of $Y$, and hence $P_i^xP_j^{x'}=0$ for $i\ne j$.  Therefore
\[
   P_i^x=P_i^x\sum_j P_j^{x'}=P_i^xP_i^{x'}
   \quad\text{and}\quad
   P_i^{x'}=\sum_j P_j^xP_i^{x'}=P_i^xP_i^{x'},
\]
so $P_i^x=P_i^{x'}$.  Since $X$ is weakly connected, $P_i^x$ is independent of $x$; write it as $P_i$.

Because $\sum_iP_i=\Id$, some $P_i$ is nonzero.  Restricting the projections $E_{x,y}$, $y\in V(Y_i)$, to the nonzero space $\operatorname{Ran}(P_i)$ gives a quantum homomorphism $X\toq Y_i$.
\end{proof}

\begin{corollary}\label{cor:dicycles-q-classical}
Let $D$ and $D'$ be finite disjoint unions of clockwise directed cycles.  Then
\[
   D\toq D'
   \quad\Longleftrightarrow\quad
   D\to D'.
\]
\end{corollary}

\begin{proof}
Let $D=\bigsqcup_i \Cyc_{n_i}$ and $D'=\bigsqcup_j\Cyc_{m_j}$.  If $D\toq D'$, then by Lemma \ref{lem:connected-to-component}, each component $\Cyc_{n_i}$ quantumly maps to some $\Cyc_{m_j}$.  Proposition \ref{prop:cycles} implies $m_j\mid n_i$, hence $\Cyc_{n_i}\to\Cyc_{m_j}$.  Thus $D\to D'$.  The converse is immediate, since every classical homomorphism is a quantum homomorphism.
\end{proof}

Let $\DiCycles$ denote the class of finite disjoint unions of clockwise directed cycles, ordered by the existence of a directed graph homomorphism.  The classical homomorphism order on $\DiCycles$ is countably universal \cite{FialaHubickaLong}.  Therefore:

\begin{theorem}\label{thm:directed-universal}
The quantum homomorphism quasi-order of finite directed graphs is countably universal.  Consequently, the quotient partial order is countably universal.
\end{theorem}

\begin{proof}
By Corollary \ref{cor:dicycles-q-classical}, the quantum order and the classical order agree on $\DiCycles$.  Since the classical order on $\DiCycles$ is countably universal, the quantum homomorphism quasi-order of finite directed graphs contains a countably universal induced suborder.  The if-and-only-if statement in Corollary~\ref{cor:dicycles-q-classical} also shows that this embedding descends to the quotient partial order.
\end{proof}

\section{A quantum endpoint-forcing indicator}

Let $J$ be an undirected graph with two distinguished vertices $a,b$.  We regard $a,b$ as ordered terminals.

\begin{definition}\label{def:replacement}
Let $H$ be an irreflexive directed graph.  The undirected graph $H*J$ is obtained as follows.  For every arc $u\to v$ of $H$, take a fresh copy $J_{uv}$ of $J$, identify the terminal $a$ of this copy with $u$, and identify the terminal $b$ with $v$.  The nonterminal vertices of the copies are all kept distinct.  If $H$ has isolated vertices, they remain as isolated vertices in $H*J$.
\end{definition}

\begin{figure}[htbp]
\centering
\begin{tikzpicture}[scale=0.75,
  conn/.style={circle,draw,fill=orange!18,inner sep=1.5pt,minimum size=8pt},
  Aroot/.style={circle,draw,fill=blue!12,inner sep=1.5pt,minimum size=6pt},
  Broot/.style={circle,draw,fill=green!18,inner sep=1.5pt,minimum size=6pt},
  lab/.style={font=\small},
  arr/.style={-{Latex[length=2mm]},thick}
]
\node[conn] (u) at (0,0) {$u$};
\node[conn] (v) at (1.5,0) {$v$};
\node[conn] (w) at (3.0,0) {$w$};
\draw[arr] (u)--(v);
\draw[arr] (v)--(w);
\node[lab] at (1.5,-0.65) {arcs in $D$};
\draw[arr] (3.7,0)--(4.7,0);
\node[lab,above] at (4.2,0.08) {$\Phi$};
\node[conn] (U) at (5.4,0) {$u$};
\node[Aroot] (a1) at (6.15,0) {$A$};
\node[Aroot] (a2) at (6.9,0) {$A$};
\node[Broot] (b1) at (7.65,0) {$B$};
\node[Broot] (b2) at (8.4,0) {$B$};
\node[Aroot] (a3) at (9.15,0) {$A$};
\node[conn] (V) at (9.9,0) {$v$};
\foreach \x/\y in {U/a1,a1/a2,a2/b1,b1/b2,b2/a3,a3/V}{\draw (\x)--(\y);}
\node[Aroot] (a4) at (10.65,0) {$A$};
\node[Aroot] (a5) at (11.4,0) {$A$};
\node[Broot] (b3) at (12.15,0) {$B$};
\node[Broot] (b4) at (12.9,0) {$B$};
\node[Aroot] (a6) at (13.65,0) {$A$};
\node[conn] (W) at (14.4,0) {$w$};
\foreach \x/\y in {V/a4,a4/a5,a5/b3,b3/b4,b4/a6,a6/W}{\draw (\x)--(\y);}
\draw[dashed,rounded corners] (5.25,-0.35) rectangle (10.05,0.35);
\draw[dashed,rounded corners] (9.75,-0.35) rectangle (14.55,0.35);
\node[lab] at (7.65,-0.75) {$J_{uv}$};
\node[lab] at (12.15,-0.75) {$J_{vw}$};
\node[lab] at (10.1,0.65) {undirected copies with ordered terminals};
\end{tikzpicture}
\caption{The ordered replacement construction.  Each arc $u\to v$ of a directed graph $H$ is replaced by a copy of the ordered indicator $J(a,b)$, with $a$ identified with $u$ and $b$ identified with $v$.  The resulting graph $H*J$ is undirected, but the ordered terminals of each copy remember the direction of the original arc.}
\label{fig:replacement-construction}
\end{figure}
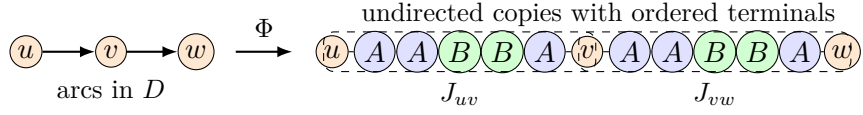

\begin{definition}\label{def:endpoint-forcing}
An ordered indicator $J(a,b)$ is \emph{quantum endpoint-forcing} if, for every irreflexive directed graph $H$, every quantum homomorphism $F:J\toq H*J$, represented by projections $F_{x,y}$ with $x\in V(J)$ and $y\in V(H*J)$, satisfies
\[
   F_{a,p}F_{b,q}=0
\]
for all pairs $(p,q)$ which are not the ordered terminal pair of a copy $J_{uv}$, i.e. for all pairs not of the form $(u,v)$ with $u\to v$ an arc of $H$.
\end{definition}

Thus quantum endpoint forcing is a joint-support statement: it does not assign images to $a$ and $b$, but rules out simultaneous support of the terminal projections on illegal ordered pairs.  The rest of this section constructs a finite indicator with this property.

\subsection{Color gadgets}

The terms $A$-root, $B$-root and $E$-root are only names for vertices carrying the following ordinary graph gadgets.

\paragraph{The $A$ gadget.}
An $A$-root $\alpha$ is the apex of a cone over a private $5$-cycle.  Thus there are private vertices $c_0,c_1,c_2,c_3,c_4$ forming a cycle $c_0c_1c_2c_3c_4c_0$, and $\alpha$ is adjacent to every $c_i$.

\paragraph{The $B$ gadget.}
A $B$-root $\beta$ has two private adjacent $A$-roots $\alpha_1,\alpha_2$, and is adjacent to both of them.  Thus $\beta\alpha_1$, $\beta\alpha_2$ and $\alpha_1\alpha_2$ are edges; the two vertices $\alpha_1,\alpha_2$ each carry their own private $A$ gadgets.

\paragraph{The $E$ gadget.}
An $E$-root $\epsilon$ has two private adjacent $B$-roots $\beta_1,\beta_2$, and is adjacent to both of them.  Thus $\epsilon\beta_1$, $\epsilon\beta_2$ and $\beta_1\beta_2$ are edges; the vertices $\beta_1,\beta_2$ each carry their own private $B$ gadgets.

Whenever these gadgets are attached in the construction below, all private vertices are fresh and no extra edges are added.  Figure~\ref{fig:color-gadgets} displays the rooted gadgets explicitly.

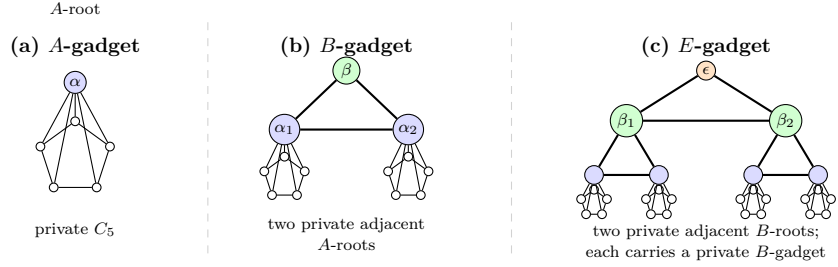
\begin{figure}[htbp]
\centering
\begin{tikzpicture}[scale=0.78,every node/.style={transform shape},
    base/.style={circle,draw,fill=white,inner sep=0pt,minimum size=4.2pt},
    rootA/.style={circle,draw,fill=blue!14,inner sep=1.5pt,minimum size=9pt,font=\scriptsize},
    rootB/.style={circle,draw,fill=green!18,inner sep=1.5pt,minimum size=9pt,font=\scriptsize},
    rootE/.style={circle,draw,fill=orange!22,inner sep=1.5pt,minimum size=9pt,font=\scriptsize},
    lab/.style={font=\scriptsize},
    title/.style={font=\small\bfseries}
]
% Helper: an A-root with its private 5-cycle drawn below the root.
% #1 prefix, #2 root coordinate, #3 root label, #4 vertical drop, #5 cycle radius
\newcommand{\drawAcone}[5]{%
  \node[rootA] (#1root) at #2 {#3};
  \foreach \i in {0,...,4}{
    \node[base] (#1c\i) at ($(#1root)+(90+72*\i:#5)+(0,-#4)$) {};
    \draw (#1root)--(#1c\i);
  }
  \foreach \i/\j in {0/1,1/2,2/3,3/4,4/0}{\draw (#1c\i)--(#1c\j);}
}

% Panel (a): A-gadget
\node[title] at (0,3.35) {(a) $A$-gadget};
\drawAcone{Afull}{(0,2.75)}{$\alpha$}{1.28}{0.62}
\node[lab] at (0,0.25) {private $C_5$};
\node[lab,align=center] at (0,4.00) {$A$-root};

% Panel (b): B-gadget
\node[title] at (4.6,3.35) {(b) $B$-gadget};
\node[rootB] (Broot) at (4.6,2.95) {$\beta$};
\drawAcone{BAone}{(3.55,1.95)}{$\alpha_1$}{0.82}{0.36}
\drawAcone{BAtwo}{(5.65,1.95)}{$\alpha_2$}{0.82}{0.36}
\draw[thick] (Broot)--(BAoneroot);
\draw[thick] (Broot)--(BAtworoot);
\draw[thick] (BAoneroot)--(BAtworoot);
\node[lab,align=center] at (4.6,0.18) {two private adjacent\\$A$-roots};

% Panel (c): E-gadget
\node[title] at (10.7,3.35) {(c) $E$-gadget};
\node[rootE] (Eroot) at (10.7,2.95) {$\epsilon$};
\node[rootB] (EBone) at (9.35,2.10) {$\beta_1$};
\node[rootB] (EBtwo) at (12.05,2.10) {$\beta_2$};
\draw[thick] (Eroot)--(EBone);
\draw[thick] (Eroot)--(EBtwo);
\draw[thick] (EBone)--(EBtwo);
% The private B-gadget below beta_1
\drawAcone{EAone}{(8.80,1.18)}{}{0.48}{0.23}
\drawAcone{EAtwo}{(9.90,1.18)}{}{0.48}{0.23}
\draw[thick] (EBone)--(EAoneroot);
\draw[thick] (EBone)--(EAtworoot);
\draw[thick] (EAoneroot)--(EAtworoot);
% The private B-gadget below beta_2
\drawAcone{EAthree}{(11.50,1.18)}{}{0.48}{0.23}
\drawAcone{EAfour}{(12.60,1.18)}{}{0.48}{0.23}
\draw[thick] (EBtwo)--(EAthreeroot);
\draw[thick] (EBtwo)--(EAfourroot);
\draw[thick] (EAthreeroot)--(EAfourroot);
\node[lab,align=center] at (10.7,0.04) {two private adjacent $B$-roots;\\each carries a private $B$-gadget};

% Subtle separators between panels
\draw[gray!35,dashed] (2.25,-0.15)--(2.25,3.85);
\draw[gray!35,dashed] (7.40,-0.15)--(7.40,3.85);
\end{tikzpicture}
\caption{The three rooted graph gadgets used to implement the colors $A$, $B$ and $E$.  An $A$-root is the apex of a cone over a private $5$-cycle.  A $B$-root has two adjacent private $A$-roots in its neighbourhood, each carrying its own private $A$-gadget.  An $E$-root has two adjacent private $B$-roots in its neighbourhood, each carrying its own private $B$-gadget.  The labels and colors are only for exposition; the construction is an ordinary graph.}
\label{fig:color-gadgets}
\end{figure}

\begin{definition}[The indicator $J$]\label{def:indicator}
Let $J=J(a,b)$ be the following graph.  Start with a path
\[
   x_0-x_1-x_2-x_3-x_4-x_5-x_6,
\]
where $a=x_0$ and $b=x_6$.  Attach color gadgets to the spine vertices as follows:
\[
   x_0,x_6\text{ are }E\text{-roots},
\]
\[
   x_1,x_2,x_5\text{ are }A\text{-roots},
\]
\[
   x_3,x_4\text{ are }B\text{-roots}.
\]
Thus the spine has color sequence
\[
   E,A,A,B,B,A,E.
\]
No further edges are added.
\end{definition}

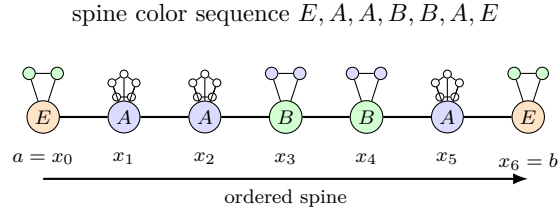
\begin{figure}[htbp]
\centering
\begin{tikzpicture}[scale=0.93,every node/.style={transform shape},
    spine/.style={circle,draw,fill=white,inner sep=1.5pt,minimum size=8pt,font=\scriptsize},
    rootA/.style={circle,draw,fill=blue!14,inner sep=1.6pt,minimum size=10pt,font=\scriptsize},
    rootB/.style={circle,draw,fill=green!18,inner sep=1.6pt,minimum size=10pt,font=\scriptsize},
    rootE/.style={circle,draw,fill=orange!22,inner sep=1.6pt,minimum size=10pt,font=\scriptsize},
    base/.style={circle,draw,fill=white,inner sep=0pt,minimum size=3.2pt},
    lab/.style={font=\scriptsize},
    title/.style={font=\small}
]
% Spine vertices
\node[rootE] (x0) at (0,0) {$E$};
\node[rootA] (x1) at (1.15,0) {$A$};
\node[rootA] (x2) at (2.30,0) {$A$};
\node[rootB] (x3) at (3.45,0) {$B$};
\node[rootB] (x4) at (4.60,0) {$B$};
\node[rootA] (x5) at (5.75,0) {$A$};
\node[rootE] (x6) at (6.90,0) {$E$};
\foreach \i/\j in {x0/x1,x1/x2,x2/x3,x3/x4,x4/x5,x5/x6}{\draw[thick] (\i)--(\j);}
\node[lab,below=4pt of x0] {$a=x_0$};
\node[lab,below=4pt of x1] {$x_1$};
\node[lab,below=4pt of x2] {$x_2$};
\node[lab,below=4pt of x3] {$x_3$};
\node[lab,below=4pt of x4] {$x_4$};
\node[lab,below=4pt of x5] {$x_5$};
\node[lab,below=4pt of x6] {$x_6=b$};

% Small attachment icons above each spine vertex: these indicate the full gadgets from Figure color-gadgets.
\newcommand{\tinyAicon}[2]{%
  \foreach \k in {0,...,4}{\node[base] (#1c\k) at ($(#2)+(90+72*\k:0.18)+(0,0.43)$) {};}
  \foreach \i/\j in {0/1,1/2,2/3,3/4,4/0}{\draw[thin] (#1c\i)--(#1c\j);}
  \foreach \k in {0,...,4}{\draw[thin] (#2)--(#1c\k);}
}
\newcommand{\tinyBicon}[2]{%
  \node[rootA,minimum size=5pt,inner sep=0pt] (#1a) at ($(#2)+(-0.20,0.62)$) {};
  \node[rootA,minimum size=5pt,inner sep=0pt] (#1b) at ($(#2)+(0.20,0.62)$) {};
  \draw[thin] (#2)--(#1a)--(#1b)--(#2);
}
\newcommand{\tinyEicon}[2]{%
  \node[rootB,minimum size=5pt,inner sep=0pt] (#1a) at ($(#2)+(-0.20,0.62)$) {};
  \node[rootB,minimum size=5pt,inner sep=0pt] (#1b) at ($(#2)+(0.20,0.62)$) {};
  \draw[thin] (#2)--(#1a)--(#1b)--(#2);
}
\tinyEicon{e0}{x0}
\tinyAicon{a1}{x1}
\tinyAicon{a2}{x2}
\tinyBicon{b3}{x3}
\tinyBicon{b4}{x4}
\tinyAicon{a5}{x5}
\tinyEicon{e6}{x6}

\draw[-{Latex[length=2.0mm]},thick] (0,-0.88)--(6.90,-0.88);
\node[lab] at (3.45,-1.15) {ordered spine};
\node[title] at (3.45,1.45) {spine color sequence $E,A,A,B,B,A,E$};
\end{tikzpicture}
\caption{The ordered indicator $J(a,b)$.  The spine is the path $x_0x_1\cdots x_6$ with terminals $a=x_0$ and $b=x_6$.  Each spine vertex carries the rooted gadget indicated by its label, as in Figure~\ref{fig:color-gadgets}.  The reverse color sequence is $E,A,B,B,A,A,E$, so the spine encodes direction.}
\label{fig:indicator-spine}
\end{figure}

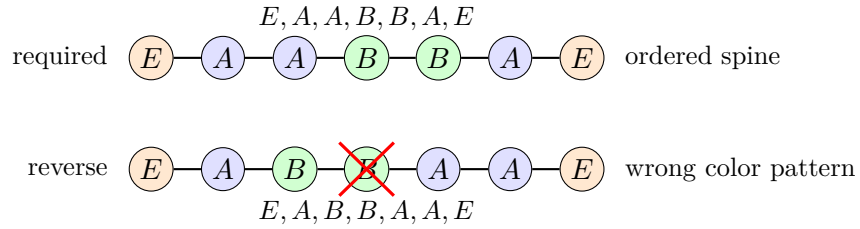
\begin{figure}[htbp]
\centering
\begin{tikzpicture}[scale=0.95,
  rootA/.style={circle,draw,fill=blue!12,inner sep=1.8pt,minimum size=9pt},
  rootB/.style={circle,draw,fill=green!18,inner sep=1.8pt,minimum size=9pt},
  rootE/.style={circle,draw,fill=orange!18,inner sep=1.8pt,minimum size=9pt},
  lab/.style={font=\small}
]
\node[rootE] (f0) at (0,1.0) {$E$};
\node[rootA] (f1) at (1.0,1.0) {$A$};
\node[rootA] (f2) at (2.0,1.0) {$A$};
\node[rootB] (f3) at (3.0,1.0) {$B$};
\node[rootB] (f4) at (4.0,1.0) {$B$};
\node[rootA] (f5) at (5.0,1.0) {$A$};
\node[rootE] (f6) at (6.0,1.0) {$E$};
\foreach \i/\j in {f0/f1,f1/f2,f2/f3,f3/f4,f4/f5,f5/f6}{\draw[thick] (\i)--(\j);}
\node[lab,left=4pt of f0] {required};
\node[lab,right=4pt of f6] {ordered spine};
\node[rootE] (r0) at (0,-0.55) {$E$};
\node[rootA] (r1) at (1.0,-0.55) {$A$};
\node[rootB] (r2) at (2.0,-0.55) {$B$};
\node[rootB] (r3) at (3.0,-0.55) {$B$};
\node[rootA] (r4) at (4.0,-0.55) {$A$};
\node[rootA] (r5) at (5.0,-0.55) {$A$};
\node[rootE] (r6) at (6.0,-0.55) {$E$};
\foreach \i/\j in {r0/r1,r1/r2,r2/r3,r3/r4,r4/r5,r5/r6}{\draw[thick] (\i)--(\j);}
\node[lab,left=4pt of r0] {reverse};
\draw[red,very thick] (2.62,-0.92)--(3.38,-0.18);
\draw[red,very thick] (2.62,-0.18)--(3.38,-0.92);
\node[lab,right=4pt of r6] {wrong color pattern};
\node[lab] at (3,1.55) {$E,A,A,B,B,A,E$};
\node[lab] at (3,-1.15) {$E,A,B,B,A,A,E$};
\end{tikzpicture}
\caption{The endpoint-forcing mechanism.  A target walk with color sequence $E,A,A,B,B,A,E$ is forced to be the ordered spine of one copy of $J$.  The reversed spine has color sequence $E,A,B,B,A,A,E$ and therefore cannot realize the required pattern.}
\label{fig:endpoint-mechanism}
\end{figure}

If $Y=H*J$, let $\Aset(Y)$, $\Bset(Y)$ and $\Eset(Y)$ be the sets of all vertices of $Y$ which are, respectively, $A$-roots, $B$-roots and $E$-roots in one of the copies of $J$.  These three sets are pairwise disjoint by construction.  Since terminals are identified with vertices of $H$, a connector vertex of $H*J$ belongs to $\Eset(Y)$ precisely when it is incident with at least one copy of $J$.

The construction gives the following structural properties.

\begin{lemma}\label{lem:structural-color}
Let $Y=H*J$ for an irreflexive directed graph $H$.  Then:
\begin{enumerate}[label=\textup{(\roman*)}]
\item If $y\notin \Aset(Y)$, then the induced graph $Y[N_Y(y)]$ is bipartite.
\item If $y\notin \Bset(Y)$, then $Y[\Aset(Y)\cap N_Y(y)]$ contains no edge.
\item If $y\notin \Eset(Y)$, then $Y[\Bset(Y)\cap N_Y(y)]$ contains no edge.
\end{enumerate}
\end{lemma}

\begin{proof}
All private gadgets are disjoint, and the only edges of $Y$ are the spine edges and the gadget edges described in Definition~\ref{def:indicator}.  We use this explicitly below.  The possible vertices of $Y$ are connector vertices, internal spine vertices, $A$-roots, cone-base vertices in private $A$ gadgets, $B$-roots, and $E$-roots.  Connector vertices may be identified across several copies of $J$, while all private gadget vertices remain private.

For (i), let $y\notin\Aset(Y)$.  We check the possible types of $y$.  If $y$ is a cone-base vertex in a private $A$ gadget, then $N_Y(y)$ consists of the corresponding $A$-root and the two neighbouring cone-base vertices on the private $C_5$; the induced graph on this neighbourhood is a path of length two.
If $y$ is an internal spine vertex which is not an $A$-root, then it is a $B$-root, so it is covered by the next case.  If $y\in\Bset(Y)$, then the edges inside $N_Y(y)$ are only the private $A$-$A$ edge belonging to the $B$ gadget carried by $y$.  If this $B$-root is also one of the two private $B$-roots in an $E$ gadget, then there is additionally the edge joining the adjacent $E$-root to the other private $B$-root of that same $E$ gadget.  These edges are disjoint, and all other neighbours of $y$ are isolated inside $N_Y(y)$.
If $y\in\Eset(Y)$, then $y$ may be a connector incident with many copies of $J$.  In each $E$ gadget carried by one of these incident copies, the two private $B$-neighbours of $y$ are joined by one edge.  The spine neighbours of $y$ coming from incoming or outgoing copies are $A$-roots, but they are isolated in $Y[N_Y(y)]$: they are adjacent to $y$ and to vertices outside $N_Y(y)$, but not to other neighbours of $y$.  Hence $Y[N_Y(y)]$ is a disjoint union of edges and isolated vertices in this case.
Finally, if $y$ is an isolated connector of $H*J$, then $N_Y(y)=\varnothing$.
In all cases every connected component of $Y[N_Y(y)]$ is an isolated vertex, an edge, or a path of length two.  Thus $Y[N_Y(y)]$ is bipartite.

For (ii), suppose two vertices $r,s\in\Aset(Y)\cap N_Y(y)$ are adjacent.  By construction, an edge between two $A$-roots is of one of two kinds: either it is the spine edge $x_1x_2$ in a copy of $J$, or it is the private $A$-$A$ edge in the gadget of a $B$-root.  The spine edge $x_1x_2$ has no common neighbour in $Y$: the neighbours of $x_1$ are the terminal $x_0$, the spine vertex $x_2$, and vertices in its private $A$ gadget, whereas the neighbours of $x_2$ are $x_1$, $x_3$, and vertices in its private $A$ gadget.  Since the private gadgets are disjoint and no loops are present, there is no third vertex adjacent to both $x_1$ and $x_2$.  By contrast, the private $A$-$A$ edge in a $B$ gadget has, as its common neighbour, exactly the $B$-root carrying that gadget.  Therefore any vertex $y$ which has two adjacent $A$-neighbours must be a $B$-root.  This proves (ii).

For (iii), the proof is analogous.  An edge between two $B$-roots is either the spine edge $x_3x_4$ in a copy of $J$, or the private $B$-$B$ edge in the gadget of an $E$-root.  The spine edge $x_3x_4$ has no common neighbour in $Y$: the neighbours of $x_3$ are $x_2$, $x_4$, and vertices in its private $B$ gadget, while the neighbours of $x_4$ are $x_3$, $x_5$, and vertices in its private $B$ gadget; the private gadgets are disjoint.  The private $B$-$B$ edge in an $E$ gadget has, as its common neighbour, exactly the $E$-root carrying that gadget.  Therefore any vertex with two adjacent $B$-neighbours must be an $E$-root.  This proves (iii).
\end{proof}

\subsection{Quantum localization of colors}

The proof of endpoint forcing has three steps: local graph structure identifies the possible colors in $H*J$, the next lemma turns this into operator localization for projections, and Proposition~\ref{prop:endpoint-forcing} converts localization into joint-support endpoint forcing.

Before proving the localization lemma, we record the intuition behind its first part.  The argument mimics the classical fact that an odd cycle admits no $2$-colouring.  If an $A$-root were sent to a vertex whose neighbourhood is bipartite, then the private $5$-cycle adjacent to that root would be forced, on the support of the corresponding projection, to behave as though it were $2$-coloured by the two parts of that neighbourhood.  The proof below writes this odd-cycle obstruction directly in terms of the projection relations defining a quantum homomorphism.

\begin{lemma}[Color localization]\label{lem:color-localization}
Let $X$ be a graph containing vertices carrying the color gadgets above, let $Y=H*J$, and let $F:X\toq Y$ be a quantum homomorphism.  Then:
\begin{enumerate}[label=\textup{(\roman*)}]
\item If $x$ is an $A$-root in $X$ and $y\notin \Aset(Y)$, then $F_{x,y}=0$.
\item If $x$ is a $B$-root in $X$ and $y\notin \Bset(Y)$, then $F_{x,y}=0$.
\item If $x$ is an $E$-root in $X$ and $y\notin \Eset(Y)$, then $F_{x,y}=0$.
\end{enumerate}
\end{lemma}

\begin{proof}
We prove the three claims in order.

\smallskip
\noindent\emph{Proof of (i).}
Let $x$ be an $A$-root and suppose $y\notin \Aset(Y)$.  Put $P=F_{x,y}$.  The vertex $x$ is adjacent to the vertices $c_0,\ldots,c_4$ of a private $5$-cycle.  By Lemma~\ref{lem:structural-color}, $Y[N_Y(y)]$ is bipartite; fix a bipartition
\[
   N_Y(y)=N_0\sqcup N_1.
\]
For $i\in\mathbb Z_5$, define
\[
   A_i=\sum_{z\in N_0}F_{c_i,z},\qquad
   B_i=\sum_{z\in N_1}F_{c_i,z},\qquad
   C_i=A_i+B_i,\qquad S_i=A_i-B_i.
\]
Because the projections $F_{c_i,z}$, $z\in V(Y)$, form a projective measurement, the operators $A_i$ and $B_i$ are orthogonal projections, $C_i$ is a projection, and
\[
   S_i^2=C_i. \tag{1}
\]
Since $xc_i$ is an edge of $X$, the edge-zero relation gives
\[
   F_{c_i,z}P=0\qquad (z\notin N_Y(y)).
\]
Equivalently, since the projections are self-adjoint, also $PF_{c_i,z}=0$ for $z\notin N_Y(y)$.  Together with completeness for the measurement at $c_i$, this gives
\[
   C_iP=P\qquad\text{for every }i. \tag{2}
\]
Because $N_0$ and $N_1$ are independent sets in $Y[N_Y(y)]$, and because $c_ic_{i+1}$ is an edge of $X$, the edge-zero relation gives
\[
   A_iA_{i+1}=0,
   \qquad
   B_iB_{i+1}=0
   \qquad (i\in\mathbb Z_5). \tag{3}
\]
Using (2) and (3),
\[
\begin{aligned}
   S_iS_{i+1}P
   &=(A_i-B_i)(A_{i+1}-B_{i+1})P\\
   &=-(A_iB_{i+1}+B_iA_{i+1})P\\
   &=-C_iC_{i+1}P.
\end{aligned}
\]
Here
\[
   C_iC_{i+1}P=C_i(C_{i+1}P)=C_iP=P,
\]
where we use (2) and associativity of operator multiplication only; no commutativity of $C_i$ and $C_{i+1}$ is being assumed.  Hence
\[
   S_iS_{i+1}P=-P. \tag{4}
\]
The same argument with $i$ and $i+1$ interchanged gives
\[
   S_{i+1}S_iP=-P. \tag{5}
\]
Let $v\in\operatorname{Ran}(P)$.  Then (1) and (2) imply
\[
   \|S_iv\|^2=\langle S_i^2v,v\rangle=\langle C_iv,v\rangle=\|v\|^2.
\]
Equations (4) and (5) imply
\[
   (S_i+S_{i+1})^2v
   =(S_i^2+S_iS_{i+1}+S_{i+1}S_i+S_{i+1}^2)v=0.
\]
The operators $S_i$ and $S_{i+1}$ are self-adjoint, since they are differences of orthogonal projections.  Hence $S_i+S_{i+1}$ is self-adjoint and $(S_i+S_{i+1})^2$ is positive semidefinite.  Taking the inner product with $v$ gives
\[
   0=\langle (S_i+S_{i+1})^2v,v\rangle
    =\|(S_i+S_{i+1})v\|^2.
\]
Thus
\[
   (S_i+S_{i+1})v=0,
\]
so $S_{i+1}v=-S_iv$ for every $i\in\mathbb Z_5$.  Going around the odd cycle gives
\[
   S_0v=-S_4v=-S_0v,
\]
and hence $S_0v=0$.  On the other hand,
\[
   \|S_0v\|^2=\langle S_0^2v,v\rangle=\langle C_0v,v\rangle=\|v\|^2,
\]
where the last equality uses $C_0P=P$ and $v\in\operatorname{Ran}(P)$.  Thus $v=0$.  Therefore $\operatorname{Ran}(P)=0$, and $P=0$.

\smallskip
\noindent\emph{Proof of (ii).}
Let $x$ be a $B$-root and suppose $y\notin \Bset(Y)$.  Put $P=F_{x,y}$.  By construction, $x$ has two adjacent $A$-neighbors $\alpha_1,\alpha_2$.

For $i=1,2$, define
\[
   C_i=\sum_{z\in \Aset(Y)\cap N_Y(y)}F_{\alpha_i,z}.
\]
We claim that $C_iP=P$.  Indeed, by completeness at $\alpha_i$,
\[
   P=\sum_{z\in V(Y)}F_{\alpha_i,z}P.
\]
If $z\notin\Aset(Y)$, then $F_{\alpha_i,z}=0$ by part (i).  If $z\in\Aset(Y)$ but $z\notin N_Y(y)$, then $F_{\alpha_i,z}P=0$ because $x\alpha_i$ is an edge of $X$ while $yz$ is not an edge of $Y$.  Hence only terms with $z\in\Aset(Y)\cap N_Y(y)$ remain, and $C_iP=P$.

Since $y\notin\Bset(Y)$, Lemma~\ref{lem:structural-color} says that $\Aset(Y)\cap N_Y(y)$ contains no edge.  As $\alpha_1\alpha_2$ is an edge of $X$, the edge-zero condition gives
\[
   C_1C_2=0.
\]
Consequently
\[
   P=C_1P=C_1C_2P=0.
\]

\smallskip
\noindent\emph{Proof of (iii).}
Let $x$ be an $E$-root and suppose $y\notin\Eset(Y)$.  Put $P=F_{x,y}$.  By construction, $x$ has two adjacent $B$-neighbors $\beta_1,\beta_2$.  For $i=1,2$, define
\[
   D_i=\sum_{z\in \Bset(Y)\cap N_Y(y)}F_{\beta_i,z}.
\]
The same argument as in (ii), now using part (ii) in place of part (i), gives $D_iP=P$ for $i=1,2$.  Since $y\notin\Eset(Y)$, Lemma~\ref{lem:structural-color} says that $\Bset(Y)\cap N_Y(y)$ contains no edge.  Since $\beta_1\beta_2$ is an edge of $X$, the edge-zero condition gives $D_1D_2=0$.  Hence
\[
   P=D_1P=D_1D_2P=0.
\]
\end{proof}

\subsection{Endpoint forcing}

\begin{lemma}\label{lem:colored-spine}
Let $Y=H*J$.  If
\[
   y_0,y_1,\ldots,y_6
\]
is a walk in $Y$ whose colors are
\[
   E,A,A,B,B,A,E,
\]
then the walk is the ordered spine of a single copy $J_{uv}$, where $u\to v$ is an arc of $H$.
\end{lemma}

\begin{proof}
Since $y_0$ has color $E$, it is a connector vertex of $Y$.  Its $A$-neighbors are of two possible types: the vertex $x_1$ in an outgoing copy, or the vertex $x_5$ in an incoming copy.  If $y_1$ were an incoming $x_5$, then $y_1$ would have no $A$-neighbor available for $y_2$, since its spine neighbors have colors $B$ and $E$, and its private gadget neighbors do not have color $A$.  Hence $y_1$ must be the $x_1$ of an outgoing copy.

In that copy, $x_1$ has a unique $A$-neighbor, namely $x_2$, so $y_2=x_2$.  The vertex $x_2$ has a unique $B$-neighbor, namely $x_3$, so $y_3=x_3$.  The vertex $x_3$ has a unique $B$-neighbor, namely $x_4$, so $y_4=x_4$.  The $A$-neighbors of $x_4$ are $x_5$ and the private $A$-roots in its $B$ gadget; the private $A$-roots have no $E$-neighbor, so the condition that $y_6$ has color $E$ forces $y_5=x_5$; see Figure~\ref{fig:endpoint-mechanism}.  Finally, $y_6$ is the terminal $x_6$ of the same copy.  Thus the walk is the ordered spine of a single copy.
\end{proof}

\begin{proposition}\label{prop:endpoint-forcing}
The ordered indicator $J(a,b)$ from Definition~\ref{def:indicator} is quantum endpoint-forcing.
\end{proposition}

\begin{proof}
Let $H$ be an irreflexive directed graph, let $Y=H*J$, and let $F:J\toq Y$ be a quantum homomorphism.  Write the spine of $J$ as
\[
   x_0=a,x_1,x_2,x_3,x_4,x_5,x_6=b.
\]
Fix $p,q\in V(Y)$.  Insert completeness at $x_1,\ldots,x_5$:
\[
\begin{aligned}
   F_{a,p}F_{b,q}
   =\sum_{y_1,\ldots,y_5\in V(Y)}
     F_{x_0,p}F_{x_1,y_1}F_{x_2,y_2}F_{x_3,y_3}
     F_{x_4,y_4}F_{x_5,y_5}F_{x_6,q}.
\end{aligned}
\]
Consider one summand, and put $y_0=p$, $y_6=q$.  If $y_0,y_1,\ldots,y_6$ is not a walk in $Y$, then some consecutive pair is a nonedge of $Y$, while the corresponding pair $x_i,x_{i+1}$ is an edge of $J$; the edge-zero condition makes the summand zero.  If some $y_i$ does not have the same color as $x_i$, then Lemma~\ref{lem:color-localization} makes the corresponding projection $F_{x_i,y_i}$ zero.  Therefore a nonzero summand can occur only from a walk of color sequence $E,A,A,B,B,A,E$.  By Lemma~\ref{lem:colored-spine}, such a walk is the ordered spine of a single copy $J_{uv}$.

Consequently, if $(p,q)$ is not the ordered terminal pair $(u,v)$ of a copy $J_{uv}$, no summand survives, and $F_{a,p}F_{b,q}=0$.
\end{proof}

\begin{remark}\label{rem:quantum-indicator-method}
Proposition~\ref{prop:endpoint-forcing} is the quantum analogue of the endpoint property used in classical indicator constructions.  Classically, one controls the pair of endpoint images $(f(a),f(b))$.  Quantumly, the meaningful statement is that the product of the two terminal projections vanishes on every illegal pair:
\[
   F_{a,p}F_{b,q}=0.
\]
This says that the endpoints $a$ and $b$ cannot simultaneously have support on an unordered, reversed, or otherwise illegal terminal pair.  In this projection-level sense, the ordered terminals of $J$ remember the direction of an arc even though the replacement graph is undirected.
\end{remark}

\section{Undirected universality}

For a finite disjoint union $D$ of clockwise directed cycles, define
\[
   \Phi(D)=D*J,
\]
where $J$ is the endpoint-forcing indicator constructed above.  This is an undirected graph.

\begin{lemma}\label{lem:indicator-reflects}
Let $D,D'$ be finite disjoint unions of clockwise directed cycles.  Then
\[
   \Phi(D)\toq \Phi(D')
   \quad\Longleftrightarrow\quad
   D\to D'.
\]
\end{lemma}

\begin{proof}
If $D\to D'$ is a classical homomorphism, then every arc $u\to v$ of $D$ is sent to an arc $f(u)\to f(v)$ of $D'$.  Mapping each copy $J_{uv}$ identically onto the copy $J_{f(u)f(v)}$ gives a classical homomorphism $\Phi(D)\to\Phi(D')$, and hence a quantum homomorphism.

Conversely, suppose $\Phi(D)\toq\Phi(D')$, represented by projections $F_{x,y}$.  Let $U=V(D)$ and $U'=V(D')$ be the connector vertices in $\Phi(D)$ and $\Phi(D')$.  In $\Phi(D')$, the ordered terminal pairs of copies of $J$ are exactly the pairs $(v,z)$ with $v,z\in U'$ and $v\to z$ an arc of $D'$.  Since every vertex of $D$ is the tail of an arc, Proposition~\ref{prop:endpoint-forcing} implies that for every connector $u\in U$ and every nonconnector $p\in V(\Phi(D'))$,
\[
   F_{u,p}=0.
\]
Indeed, choose an arc $u\to w$ of $D$ and use
\[
   F_{u,p}=F_{u,p}\sum_qF_{w,q}=\sum_qF_{u,p}F_{w,q};
\]
endpoint forcing makes every summand zero when $p$ is not a connector.

Thus, for $u\in U$ and $v\in U'$, set
\[
   Q_{u,v}=F_{u,v}.
\]
For each $u\in U$, the projections $Q_{u,v}$ sum to $\Id$.  In $\Phi(D')$, the ordered terminal pairs of copies of $J$ are exactly the pairs $(v,z)$ of connector vertices for which $v\to z$ is an arc of $D'$.  Indeed, each arc of $D'$ gives one copy of $J$, with the first terminal identified with the tail and the second terminal identified with the head.  Therefore, if $u\to w$ is an arc of $D$, endpoint forcing gives
\[
   Q_{u,v}Q_{w,z}=0
\]
whenever $v\to z$ is not an arc of $D'$.  Hence the projections $Q_{u,v}$ define a quantum homomorphism $D\toq D'$.  By Corollary~\ref{cor:dicycles-q-classical}, $D\to D'$.
\end{proof}

\begin{remark}\label{rem:general-reflection}
The preceding proof uses only the endpoint-forcing property and the fact that every connector vertex of the source is incident with a directed edge.  Thus the same argument applies more generally as follows: if $G$ is an irreflexive directed graph with no isolated vertices and $H$ is any irreflexive directed graph, then
\[
   G*J\toq H*J
   \quad\Longrightarrow\quad
   G\toq H.
\]
Indeed, the endpoint-forcing condition first shows that the projections attached to connector vertices of $G*J$ have no support on nonconnector vertices of $H*J$; if a connector is the tail of an arc one uses completeness at the head, and if it is the head of an arc one uses completeness at the tail.  The remaining projections then define a quantum homomorphism $G\toq H$.  We only need the directed-cycle case for universality, but this reflection statement indicates that the indicator is not tailored only to cycles: it can serve as a directed-to-undirected encoding mechanism whenever isolated vertices are not involved.
\end{remark}

\begin{theorem}\label{thm:undirected-universal}
The quantum homomorphism quasi-order of finite undirected graphs is countably universal.  Consequently, the quotient partial order is countably universal.
\end{theorem}

\begin{proof}
By Lemma~\ref{lem:indicator-reflects}, the map $D\mapsto \Phi(D)$ embeds the classical homomorphism order on $\DiCycles$ as an induced suborder of the quantum homomorphism quasi-order of finite undirected graphs.  Since the classical order on $\DiCycles$ is countably universal \cite{FialaHubickaLong}, the undirected quantum homomorphism quasi-order is countably universal.  The if-and-only-if statement in Lemma~\ref{lem:indicator-reflects} also shows that the embedding descends to the quotient partial order.

\end{proof}

The following corollary records the finite version as a reusable recipe.  The point is not only that the desired finite suborder exists: the same fixed undirected indicator $J$ works for every poset, and all order information is encoded in the integers $N_p$.

\begin{corollary}[Explicit finite representation]\label{cor:finite-posets}
Let $(P,\le_P)$ be a finite partial order.  Then there are finite undirected graphs $G_p$, $p\in P$, such that for all $p,q\in P$,
\[
   p\le_P q
   \quad\Longleftrightarrow\quad
   G_p\toq G_q.
\]
Consequently, the quantum homomorphism order of finite graphs contains antichains of arbitrary finite size, chains of arbitrary finite length, and induced copies of every finite partial order.  Moreover, the representation may be chosen by the explicit prime-divisibility and indicator-replacement construction described below.
\end{corollary}

\begin{proof}
Choose distinct primes $\pi_r$, one for each $r\in P$, and define
\[
   N_p=\prod_{r:\,p\le_P r}\pi_r.
\]
Then $p\le_P q$ if and only if $N_q$ divides $N_p$.  Indeed, if $p\le_P q$, then every element above $q$ is also above $p$, so the set of prime factors of $N_q$ is contained in the set of prime factors of $N_p$.  Conversely, if $N_q\mid N_p$, then the prime $\pi_q$ divides $N_p$, and hence $q$ lies in the upper set of $p$, i.e. $p\le_P q$.

Set
\[
   G_p=\Phi(\Cyc_{N_p}).
\]
By Proposition~\ref{prop:cycles} and Lemma~\ref{lem:indicator-reflects},
\[
   G_p\toq G_q
   \quad\Longleftrightarrow\quad
   \Cyc_{N_p}\to \Cyc_{N_q}
   \quad\Longleftrightarrow\quad
   N_q\mid N_p
   \quad\Longleftrightarrow\quad
   p\le_P q.
\]
The final statement follows by applying the construction to finite antichains, finite chains, and arbitrary finite posets.  Since the indicator $J$ is fixed once and for all, this gives an explicit poset-to-undirected-graph encoding inside the quantum homomorphism order.
\end{proof}

\begin{remark}
The preceding proof should be viewed as a usable construction rather than only as an existence proof.  Once the desired finite order pattern is specified, the integers $N_p$ are obtained from its upper sets, and the graphs are obtained by applying the same fixed undirected indicator to the corresponding directed cycles.  The construction therefore supplies a flexible source of examples for testing conjectures or building gadgets in the quantum homomorphism order.

The same proof gives concrete examples.  If $p_1,\ldots,p_k$ are distinct primes, then
\[
   \{\Phi(\Cyc_{p_i}):1\le i\le k\}
\]
forms a $k$-element antichain.  The graphs
\[
   \Phi(\Cyc_{2^k}),\ \Phi(\Cyc_{2^{k-1}}),\ \ldots,\ \Phi(\Cyc_2)
\]
form a chain of length $k$ in the displayed order.  For the diamond poset $\hat 0<a,b<\hat 1$ with $a$ and $b$ incomparable, one may take
\[
   N_{\hat 0}=210,
   \qquad
   N_a=6,
   \qquad
   N_b=10,
   \qquad
   N_{\hat 1}=2.
\]
Here the divisibility direction is reversed relative to the order, because
\[
   \Cyc_m\to\Cyc_n
   \quad\text{holds exactly when}\quad
   n\mid m.
\]
Thus $\Phi(\Cyc_{210})$ maps quantumly to $\Phi(\Cyc_6)$ and $\Phi(\Cyc_{10})$, both of those map quantumly to $\Phi(\Cyc_2)$, while $\Phi(\Cyc_6)$ and $\Phi(\Cyc_{10})$ are incomparable; see Fig.~\ref{fig:diamond-encoding}.
\end{remark}

\begin{figure}[htbp]
\centering
\begin{tikzpicture}[scale=1.0,
    point/.style={circle,draw,fill=white,inner sep=2pt,minimum size=16pt},
    lab/.style={font=\small},
    tiny/.style={font=\scriptsize},
    edge/.style={thick}
]
% left: diamond order
\node[point] (zero) at (0,0) {$\hat 0$};
\node[point] (a) at (-1.15,1.15) {$a$};
\node[point] (b) at (1.15,1.15) {$b$};
\node[point] (one) at (0,2.30) {$\hat 1$};
\draw[edge] (zero)--(a)--(one)--(b)--(zero);
\node[lab] at (0,2.85) {diamond order};
\node[tiny,left=7pt of zero] {$N_{\hat 0}=210$};
\node[tiny,left=7pt of a] {$N_a=6$};
\node[tiny,right=7pt of one] {$N_{\hat 1}=2$};
\node[tiny,right=7pt of b] {$N_b=10$};
% right: replacement notation
\node[lab,align=center] at (4.8,1.70) {replace cycles by indicators};
\node[lab,align=center] at (4.8,1.05) {$G_p=\Phi(\overrightarrow C_{N_p})$};
\node[lab,align=center] at (4.8,0.42) {$G_p\to_q G_q\iff p\le q$};
\draw[-{Latex[length=2.2mm]},thick] (1.9,-0.25)--(3.35,-0.25);
\end{tikzpicture}
\caption{A concrete finite-poset encoding.  The diamond order $\hat 0<a,b<\hat 1$ is represented by the divisibility pattern $210\to 6,10\to 2$.  After applying the replacement construction, $G_p=\Phi(\overrightarrow C_{N_p})$, the same comparability pattern is realized by quantum homomorphisms.}
\label{fig:diamond-encoding}
\end{figure}
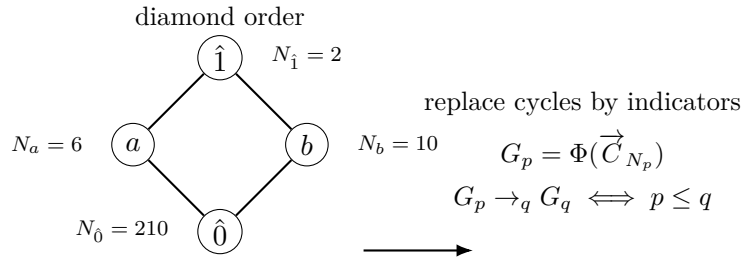

Combining Theorems~\ref{thm:directed-universal} and~\ref{thm:undirected-universal} proves Theorem~\ref{thm:main}.

\section{Concluding remarks}

The proof deliberately avoids quantum cores, quantum homomorphism images, measurement graphs, and universal $C^*$-algebras.  The endpoint-forcing indicator is a local finite gadget whose proof uses only the projection relations in the definition of a quantum homomorphism, namely completeness of the measurements and the edge-zero relations.  Its role is to replace the classical pointwise endpoint condition by the projection-level condition that illegal products $F_{a,p}F_{b,q}$ vanish.

The construction also shows that ordered indicators are natural objects for encoding directed information into undirected quantum homomorphism order.  The terminals of the indicator need not be interchangeable; direction is encoded by the asymmetric color sequence on the spine.  In the present application this gives a concrete way to transfer the directed-cycle universality construction to finite undirected graphs without introducing extra quantum arrows.

The explicit finite representation may be useful as a source of examples and gadgets.  Whenever a construction requires a prescribed finite pattern of quantum-homomorphic comparability and incomparability, the recipe in Corollary~\ref{cor:finite-posets} produces such a pattern by finite undirected graphs.

Several natural questions remain.  One can ask for a systematic quantum analogue of classical indicator and arrow constructions, for smaller or more restricted endpoint-forcing indicators, and for universality inside sparse classes such as planar or bounded-degree graphs.  The reflection statement in Remark~\ref{rem:general-reflection} also suggests asking when a replacement $G\mapsto G*J$ reflects, or preserves, quantum homomorphisms beyond the directed-cycle setting.  Finally, it would be interesting to know whether endpoint forcing has robust or model-dependent variants, for instance for approximate, commuting-operator, or non-signalling homomorphisms.

\section*{Declaration of generative AI and AI-assisted technologies in the writing process}
During the preparation of this work the author used ChatGPT to assist with drafting, rewriting, and language editing. After using this tool, the author reviewed and edited the content as needed and takes full responsibility for the content of the manuscript.

\bibliographystyle{abbrvnat}
\bibliography{references}

\end{document}